\documentclass{amsart}
\usepackage{amsmath,amsthm, amssymb, latexsym,}
\usepackage{cases}
\usepackage{mathrsfs}
\usepackage{pifont}
\usepackage{amsfonts}
\usepackage{stmaryrd}
\usepackage{color}

\usepackage[english]{babel}
\usepackage[latin1]{inputenc}

\usepackage[all]{xy}
\newtheorem{theorem}{Theorem}[section]
\newtheorem{lemma}[theorem]{Lemma}
\newtheorem{corollary}[theorem]{Corollary}
\newtheorem{proposition}[theorem]{Proposition}

\theoremstyle{definition}
\newtheorem{definition}[theorem]{Definition}

\theoremstyle{remark}
\newtheorem{remark}[theorem]{Remark}

\numberwithin{equation}{section}

\newcommand{\on}[1]{\operatorname{#1}}
\newcommand{\abs}[1]{\lvert#1\rvert}

\newcommand{\un}[1]{\underline{#1}}

\newcommand{\xg}{\backslash} 

\newcommand{\lieg}{\mathfrak{g}} 
\newcommand{\lieh}{\mathfrak{h}} 
\newcommand{\liet}{\mathfrak{t}} 
\newcommand{\liel}{\mathfrak{l}} 
\newcommand{\lieb}{\mathfrak{b}}

\newcommand{\p}{\partial} 

\newcommand{\co}[1]{\underline{#1}} 

\newcommand{\ds}{\oplus} 
\newcommand{\bds}{\bigoplus} 

\newcommand{\fz}{\mathbb{Z}} 
\newcommand{\fn}{\mathbb{N}} 
\newcommand{\fk}{\mathbb{K}} 
\newcommand{\ff}{\mathbb{F}} 
\newcommand{\bb}[1]{\mathbb{#1}} 

\newcommand{\ra}{\rightarrow} 
\newcommand{\sur}{\twoheadrightarrow} 
\newcommand{\mt}{\mapsto} 

\newcommand{\iso}{\simeq} 

\newcommand{\pr}[1]{#1^{\prime}}

\begin{document}

\title[Borel subalgebras]{Borel subalgebras of Cartan Type Lie Algebras}



\author{Ke Ou}
\address{Department of Statistics and Mathematics, Yunnan University of Finance and Economics,
	Kunming, 650221, China.}
\email{justinok1311@hotmail.com}


\author{Bin Shu}
\address{Department of Mathematics, East China Normal University,
	Shanghai, 200241, China.}
\email{bshu@math.ecnu.edu.cn}





\begin{abstract}
Let $ W(n) $ be Jacobson-Witt algebra over algebraic closed field $ \fk $ with positive characteristic $ p>2. $ It is difficult to classify all Borel subalgebras of $ W(n) $ or non-classical restricted simple Lie algebras. The present paper and \cite{S7} study two kinds of subalgebras which are easily to understand and highly related to Borel subalgebras. 

In \cite{S7}, the last author investigates a class of special Borel subalgebras of $W(n)$ which is called homogeneous Borel subalgebras. The present paper focuses on subalgebras of $ W(n) $ which are related to Borel subalgebras such that firstly, they could be trigonalizable; and secondly, they essentially belong to the ones investigated in \cite{S7}. In this paper, the conjugation classes of these subalgebras and representative for each class will be determined. Then some properties such as filtration and dimension will be investigated.
\end{abstract}

\maketitle
\setcounter{tocdepth}{1}

\section{Introduction}
If $ A $ is a non-associative algebra of finite dimension over a (commutative) field, its maximal solvable subalgebras will be called the {\em Borel subalgebras of A}. In this paper, we study the Borel subalgebras of non-classical simple Lie algebras in positive characteristic with certain conditions.

Over last few decades, Borel subalgebras has been investigated and generalised by many authors.
In \cite{Gr}, J.Green investigates so-called Borel
subalgebras of the Schur algebra associated. In \cite{DR}, J.Du and H.Rui investigate the existence of the Borel type subalgebras of a $ q \on{-Schur} ^m $ algebra. In \cite[Appendix]{Ko1}, the notion of Borel subalgebras for a quasi-hereditary algebra has been introduced by Scott. Then in \cite{Ko1,Ko2,Ko3}, S.Konig introduces and investigates the exact Borel subalgerbas and strong exact Borel subalgebras for an quasi-hereditary algebra. In \cite{DP}, I. Dimitrov and I.Penkov study the Borel subalgebras of the Lie algebra $ \frak{gl}(\infty) $ of finitary infinite matrices.

As we known, Borel subalgebras of a Lie algebra play an important role in the structure and representation theory. However, there is less study
on them for non-classical restricted simple Lie algebras. We neither know the number of conjugacy classes of Borel subalgebras nor what kind of
role the Borel subalgebras play in the representation theory 
although their Cartan subalgebras are well-known (cf. \cite{St2}).

It seems very difficult to classify all Borel subalgebras of non-classical restricted simple Lie algebras. Therefore,  \cite{S7} and this paper study two kinds of subalgebras which are easier to understand and highly related to Borel subalgebras.
In \cite{S7}, the last author investigates a class of special Borel subalgebras, the so-called homogeneous Borel subalgebras. This paper is a continuous of \cite{S7} to focus on completely solvable subalgebras of non-classical restricted simple Lie algebras.

Among non-classical simple Lie algebras, the Jacobson-Witt algebras $W(n)$ are primary, which are also the main objects in present paper. Since 1960's, they have been extensively studied (cf.  \cite{De1, LN2, LN3, S5, SY2} etc.).

The theorem of Borel(\cite{Bo}) and Morozov(\cite{Mo}) asserts that the Borel subalgebras of a semisimple complex Lie algebra $ \lieg $ are conjugate with respect to automorphism group. The same result is also true for a classical simple Lie algebra over an algebraically
closed field of prime characteristic with some mild restriction on the characteristic (\cite{Hum}), as well as alternative and Jordan algebras(\cite{Pe}).

However, this conjugation phenomenon will fail for Borel subalgebras of  non-classical Lie algebras over positive characteristic field.  There are $ (n+1) $ conjugacy classes of homogeneous Borel subalgebras of  $W(n)$ (cf. \cite{S7}). As a corollary, there are at least $ (n+1) $ conjugation classes of Borel subalgebras of $ W(n). $

Based on \cite{S7}, the present paper focuses on a special class of completely solvable subalgebras such that firstly, they could be trigonalizable; and secondly, they essentially belong to the ones investigated in \cite{S7}. In this paper, the conjugation classes of these subalgebras and representatives for each class will be determined. Then some properties such as filtration and dimension will be investigated.

We collect such subalgebras, endow the set of them with a variety structure and establish an analogy of classical Springer theory for Cartan type Lie algebras in other papers.

\section{Preliminaries}


Entire the whole paper, denote $ I=\{ 0,1,\cdots,p-1 \} $ and we always assume the ground filed $\mathbb{K}$ to be algebraically closed of odd characteristic $p>0.$ Unless mentioned otherwise, all vector spaces
are assumed to be finite-dimensional . Given a restricted Lie algebra $ (\lieg, [p]), $ we
have an adjoint group $ G := \on{Aut}_p(\lieg)^{\circ}, $ the identity component of its restricted automorphism group. The term {\em filtration} stands a descending filtration.

\subsection{Graded dimension of a graded algebra}
\begin{definition}\label{grading assumption}
	 For a given graded algebra $ (\liel,(\liel_{n})), $ set $ G:=\text{Aut}_p(\liel)^{\circ}. $ We call $ (\liel,(\liel_n), G, G_0, U ) $ satisfies {\em graded assumption} if $ G=G_0 \ltimes U $ where $ G_0 $ consists of homogeneous automorphisms and $ u\cdot x -x\in \liel_{\geq i+1} $ for all $ u\in U, x\in \liel_{i}.$
\end{definition}

\begin{lemma}\label{filtration}
	Let $ (\liel,(\liel_n)_{n\in\fz}) $ be a graded Lie algebra such that $ (\liel,(\liel_n),G,G_0,U) $ satisfying graded assumption, and $\lieb$ be a graded subalgebra of $ \liel. $ Then every subalgebra conjugates to $\lieb$ is filtered, namely $ C:=g\cdot\lieb$ admits a filtration structure $ (C_{(n)})_ {(n\in\fz)} $ for all $ g\in G. $ Moreover, $ (C_{(n)}) $ is restricted if all $ \liel, $ $ (\liel_{(n)}) $ and $ \lieb $ are restricted and $ g $ is a restricted automorphism.
\end{lemma}
\begin{proof}
	Denote $  \lieb_i:= \lieb\cap \liel_i, $ then $  \lieb=\ds_{i\in\fz}  \lieb_i. $ Moreover, $  \lieb_{\geq n}:=\ds_{i\geq n}  \lieb_i= \lieb\cap \liel_{\geq n} $ for $ n\in\fz $ defines a filtered structure for $  \lieb. $
	
	Set $ C_{(n)}:=g\cdot  \lieb_{\geq n}. $ One can check the lemma by definitions.
\end{proof}

\begin{definition}\label{gdim}
	Let $ (\liel,(\liel_i)_{i\in\fz}) $ be a $ \fz $-graded Lie algebras with finite dimensional homogeneous spaces. The \textit{graded dimension associated with }$ L $ is defined by:
	$$ 
	\mathsf{gdim}(\liel,(\liel_i)):= \sum_{i\in\fz}(\mathsf{dim} \liel_i) t^i \in \fn[[t,t^{-1}]].
	$$
	
	Moreover, if $ (A,(A_{(n)})_{n\in\fz}) $ is a filtered Lie algebra such that $ \mathsf{gr}(A) $ has finite dimensional homogeneous spaces, we can also define its graded dimension as
	$$
	\mathsf{gdim}(A,(A_{(n)})):=\mathsf{gdim}(\mathsf{gr}(A),(\mathsf{gr}(A)_n)).
	$$
\end{definition}

\begin{lemma}\label{gdim1}
	Let $ (\liel,(\liel_i)_{(i\in\fz)}) $ be a $ \fz $-graded Lie algebras with finite dimensional homogeneous spaces. Suppose $ (\liel,(\liel_i)_{(i\in\fz)}, G,G_0, U) $ satisfies graded assumption, and $ \lieb $ is a graded subspace of $ \liel. $ Then 
		$
		\mathsf{gdim}( g\cdot \lieb)=\mathsf{gdim}( \lieb)
		$ for all $ g\in G. $
		
		Namely, graded dimension is $ G $-invariant.
\end{lemma}
\begin{proof}
By the assumptions of $ (G,G_0,U), $ $ \mathsf{gr}(u\cdot g_0\cdot \lieb)=g_0\cdot  \lieb ,\ \forall g_0\in G_0,u\in U. $
\end{proof}

\subsection{Completely solvable subalgebras}
If $ L $ is arbitrary Lie algebra over $ \mathbb{K}, $ define
$ L^{[n+1]}:=[L,L^{[n]}] ,\ L^{(n+1)}:=[L^{(n)},L^{(n)}] $ and $ L^{[0]}=L^{(0)}=L\text{  for all } n\in\fn.  $
Then $ L^{[1]}=L^{(1)}. $ We say that $ L $ is nilpotent (resp. solvable) if $ L^{[n]}=0 $ (resp. $ L^{(n)}=0 $) for some $ n. $

\begin{definition}
	A Lie algebra $ \lieb $ is called completely solvable if $ [ \lieb,  \lieb] $ is nilpotent (cf. \cite{SF}).
\end{definition}
The importance of completely solvable Lie algebra comes from the following feature.

\begin{lemma} \cite[Lemma 8.6]{SF}
	Let $ (\liel,[p]) $ be a restricted, completely solvable Lie algebra of finite dimension and such that $ C(\liel)^{[p]}=0. $
	\begin{enumerate}
		\item If $ \liet $ is a maximal torus of $ \liel $, then $ \liel=\liet\ds \on{rad}_p(\liel). $
		\item Every irreducible restricted representation of $ \liel $ is one-dimensional.
	\end{enumerate}
\end{lemma}

\begin{remark}
	If a solvable subalgebra can be embedded into a Lie algebra of classical type, it must be complete. This will fail for Cartan type Lie algebras. There are examples of subalgebras which are solvable other than completely solvable.
\end{remark}

\subsection{Basics on Jacobson-Witt algebra}

Let $W(n)$ be Jacobson-Witt algebra over $ \fk $, which is the derivation algebra over the truncated polynomial ring $A(n)=\mathbb{K}[x_1,\cdots,
x_n]\slash (x_1^p,\cdots,x_n^p), $ namely, 
$$W(n)=\bigoplus_{i=1}^n A(n)\partial_i,$$
where $ \p_i(x_j)=\delta_{i,j}, $ for all $ i,j=1,\cdots,n. $

Set $ G:=\on{Aut}_p(W(n))^{\circ}. $ There is an isomorphism from $ \on{Aut}(A(n)) $ to $ G $ by sending $ \psi $ to $ \Psi_\psi $ via $ \Psi_\psi: D \mt  \psi\circ D\circ \psi^{-1} $ (cf. \cite{Wi}). 

For each automorphism $ \psi\in \on{Aut}(A(n)), $ set  $\tilde{\psi}_i:=\psi(x_i)\in A(n)_{\geq 1} (i=1,\cdots,n) $ and $ J(\psi):=\big( \p_i(\tilde{\psi}_j) \big)_{n\times n}\in \on{Mat}_{n\times n}(A(n)). $ Then $ \psi $ is determined by $ \tilde{\psi}_1,\cdots,\tilde{\psi}_n $ and $ \psi\in G $ if and only if $ J(\mu) $ is invertible.

We list some basic material on $ W(n) $ in the following proposition.
\begin{proposition}
	Keep notations as above, then we have
	\begin{enumerate}
		\item $ \{ x^{\un{a}}\p_i \mid  1\leq i\leq n; \un{a}\in I^n \} $ is a basis of $ W(n) $ over $ \fk. $ 
		In particular, $ \mathsf{dim}_{\fk}W(n)= np^n. $
		
		\item $ [D, fE]=D(f)E + f[D,E], $ for all $ D,E\in W(n), f\in A(n). $
		
		\item There is a so-called standard  grading structure of $ W(n): $
		$$ W(n) = \ds_{i=-1}^sW(n)_i, s=n(p-1)-1, \text{ where } $$ 
		$$ W(n)_i:=\mathsf{span}_\fk\{ x^{\un{a}}\p_l \mid  \abs{\un{a}}=i+1 \}. $$
		
		\item Suppose $ n\geq 2. $ Then $ W(n) $ is generated by $ W(n)_{-1}\ds W(n)_1. $
		
		\item $ A(n) $ is a $ W(n) $-module by $ f\p_i\cdot g:=f\p_i(g) $ for $ f\p_i\in W(n),$  $ g\in A(n). $ We call the corresponding representation $ \rho :W(n)\ra  \frak{gl}(A(n)) $ the natural representation of $ W(n). $
		
		\item The representation $ \bar{\rho}: W(n)_0\ra  \frak{gl}(A(n)_1) $ induces from $ \rho $ is an isomorphism. Remind that $ W(n)_0= \langle x_i\p_j \mid  i,j=1,\cdots,n \rangle, $ $ A(n)_1 = \langle x_1,\cdots,x_n \rangle, $ and
		$
		\bar{\rho}(x_i\p_j) =E_{ij}, $  which sends $ x_i $ to $ x_j $.
		
	\end{enumerate}
\end{proposition}

\begin{remark}
	If $ p=3,\ W(1)\iso sl(2) $ as restricted Lie algebras. We will omit this very special case throughout this paper.
\end{remark}
	

For convenience, we fix the following notations.

\begin{tabular}{l}
$ A(n)_i:=\{$homogeneous truncated polynomials of degree $ i $ in   $ A(n)  \}, $\\

$ A(n)_{\geq i}:=\sum_{j\geq i}A(n)_j, $\\



$ W(n)_{\geq i}:=\sum_{j\geq i}W(n)_j. $
\end{tabular}

For more details of $ W(n), $ reader refers to \cite[chapter 4]{SF}.

\subsection{Automorphisms of $ W(n) $.}


\begin{lemma}{\rm (}\cite{Wi}{\rm )}
Let $ \lieg = W(n) $ over $ \bb{K} $ with $ p \geq 3 $ (unless $ n = 1 $ with assumption $ p > 3 $). The following statements hold.

\begin{enumerate}
\item $ G=\on{Aut}(\lieg) $ coincides with the adjoint group $ \on{Aut}_p(\lieg)^{\circ}. $ Hence it is a connected algebraic group.
\item $ G $ is a semi-direct product $ G = G_0 \ltimes U, $ where $ G_0\iso \on{GL}(n,\bb{K}) $ consists of those automorphisms preserving the $ \fz $-grading of $ \lieg $, and
$$
U = \{g \in G\big| (g - Id_{\lieg})(\lieg_{\geq i}) \subset \lieg_{\geq i+1}\}
$$
\end{enumerate}
\end{lemma}

\begin{remark}
$ (\lieg, (\lieg_{i})_{i\in\fz}, G, G_0, U) $  satisfies graded assumption as definition \ref{grading assumption}.
\end{remark}

 The following lemma is an algorithm to compute automorphisms.
\begin{lemma}\label{arthimatics of automorphisms}
Keep assumptions and notations as above, for all $ \mu\in G, $
$$
\Psi_{\mu}\left(
\begin{matrix}
\p _1\\
\vdots \\
\p _n
\end{matrix}\right)
=\left(
\begin{matrix}
\Psi_{\mu}(\p _1)\\
\vdots \\
\Psi_{\mu}(\p _n)
\end{matrix}\right)
=J(\mu)^{-1}\left(
\begin{matrix}
\p _1\\
\vdots \\
\p _n
\end{matrix}\right).
$$
\end{lemma}
\begin{proof}
It is a direct computation that
$ \big( \p_i(\tilde{\mu}_j) \big)(\Psi_{\mu}(\p _i)(x_j))=I_n. $ Hence lemma holds.
\end{proof}

\subsection{Maximal Torus of $ W(n) $}
A torus $ \liet $ is an abelian restricted subalgebra consisting of semisimple elements, i.e. $ X\in (X^{[p]})_p $ for all $ X\in \liet $, where $ (X^{[p]})_p $ denotes for the restricted subalgebra generated by  $ X^{[p]} $ (see \cite{SF}).
According to Demuskin's result \cite{De1}, we have the following conjugacy property for maximal torus of $ W(n). $
\begin{theorem}
	Let $ \lieg = W(n). $ Then the following statements hold.
	\begin{enumerate}
		\item Two maximal torus $ \liet, \pr{\liet} $ belong to the same $ G $-orbit if and only if $$ \on{dim}(\liet\cap \lieg_{\geq0})= \on{dim}(\pr{\liet}\cap \lieg_{\geq0}). $$
		\item There are $ (n+1) $ conjugacy classes for the maximal torus of $ \lieg $. Each maximal torus of $ \lieg $ is conjugate to one of
		\[ \liet_r=\sum_{i=1}^{n} \fk z_i\p_i,\ r=0,\cdots,n, \]
		where $ z_i=x_i $ for $ i=1,\cdots,n-r, $ and $ z_i=1+x_i $ for $ i=n-r+1,\cdots,n. $
	\end{enumerate}
\end{theorem}
Theorem 2.2. 
We call these $ \liet_r $ the standard maximal torus of $ W(n) $.

\subsection{Gradings associated with $ \liet_r $}
Note that  $ A(n) $ can be presented as the quotient algebra $ \fk[T_1;\cdots; T_n]/ (T_1^p-1,\cdots, T_n^p-1). $ Denote the image of $ T_i $ by $ y_i $ in the quotient algebra. Then we can write $ A(n) $
as $ \fk[y_1,\cdots,y_n]. $ Comparing with the notations in section 1.3, we have $ y_i=1+x_i,\ i=1,\cdots,n. $ 

Now, fix $ r=0,1,\cdots ,n, $ $ A(n) $ can be presented as a truncated polynomial
\[ \fk[z_1,\cdots, z_{n-r};z_{n-r+1},\cdots, z_n] \]
with generator $ z_i=x_i,\ z_j=y_j,\text{ where } i=1,\cdots,n-r;j=n-r+1,\cdots,n, $ and defining relations:
\[ [x_i,x_{\pr{i}}]=[y_i,y_{\pr{i}}] = [x_i,y_j] =x_i^p =y_j^p-1 =0. \]

$ \text{Moreover, } W(n)=\sum_{i=1}^{n}\sum_{\co{a}(i)}\fk Z^{\co{a}(i)}\p_i $ 
where $ Z^{\co{a}(i)}=z_1^{a_{i1}} z_2^{a_{i2}}\cdots z_n^{a_{in}} $ with $ \co{a}(i) = (a_{i1} ,\cdots, a_{in})\in I^n. $ 

There is a space decomposition as below, called
$ \fz(\liet_r) $-grading:
\[ W(n)=\bds_s W(n)_s^{(\liet_r)},\ \text{with}\
W(n)_s^{(\liet_r)} = \langle Z^{\co{a}(i)}\p_i\mid \abs{\co{a}(i)} = s+1, i=1,\cdots,n \rangle. \]
In fact, every homogenous space $ W(n)_s^{(\liet_r)} $ is a $ \liet_r $-module. For the case $ r = 0, $ the $ \fz(\liet_0) $-graded structure coincides with the standard graded structure of $ W(n), $ namely $ W(n)_s = W(n)_s^{(\liet_0)} $ for all $ s. $


Let $ \lieh $ be a subalgebra of $ W(n). $ Call $ \lieh $ a $ \fz(\liet_r) $-\textit{graded subalgebra} if $ \lieh=\sum_i \lieh_i^{(\liet_r)}, $ where $ \lieh _i^{(\liet_r)} =\lieh \cap W(n)_i^{(\liet_r)}. $

We refine $ \fz(\liet_r) $-grading. If a subalgebra $ \lieh $ is $ \fz(\liet_r) $-graded, we set for every $ \alpha\in I^n $
with $ \abs{\alpha}=s+1, \lieh_{\alpha}^{(\liet_r)} = \{ v\in \lieh_s^{(\liet_r)} \mid \text{ad} (z_i\p_i)(v)=\alpha_iv, i=0,\cdots,n \}. $ Then we call $ \lieh $ is $ \liet_r $-\textit{graded} if $ \lieh $ is $ \fz(\liet_r) $-graded and \[ \lieh=\sum_{\alpha\in I^n}\lieh_{\alpha}^{(\liet_r)}. \]

Moreover, $ \lieh $ is called {\em torus graded} if $ \lieh $ contains a maximal torus of $ W(n) $ and for every maximal torus $ \liet\subseteq\lieh, $ $ \varphi\cdot\lieh $ is $ \liet_r $-graded where $ \liet_r=\varphi\cdot \liet $ for some $ \varphi\in G $ and $ r=0,\cdots,n. $


\section{Completely solvable subalgebras of Jacobson-Witt algebras}


\subsection{Homogeneous Borel subalgebras of restricted Lie algebras}
B.Shu introduces homogeneous Borel subalgebras of $ W(n) $ in \cite{S7} which is by definition a torus graded and maximal solvable subalgebras of $ W(n). $
He proves that there are $ n+1 $ conjugation classes of homogeneous Borel subalgebras of $ W(n) $ with representatives $ B_0,\cdots,B_n. $ The definition of $ B_i (i=0,\cdots,n) $ refers to \cite{S7}.

We mention here that $ B_r $ is not completely solvable except $ r=0. $ For example,
	$ B_1\subseteq W(2) $ is a subspace spanned by
	$$
	S:= \{ \p_1,\ x_1\p_1,\ \p_2,\ x_1\p_2, \cdots,x_1^{p-1}\p_2,\ x_2\p_2,\ x_1x_2\p_2 \cdots, x_1^{p-1}x_2\p _2 \}.
	$$
	One can check the following statements.
	
	\begin{enumerate}
		\item $ B_1 $ is a maximal solvable subalgebra  (\cite{S7}).
		
		\item $ \p_2\in (B_1^{[1]})^{[n]} $ for all $ n\geq 0. $ Therefore, $ \liel^{[1]} $ is not nilpotent and $ \liel $ is not completely solvable.
		
		\noindent In fact, since $ \{x_2\p_2, \p_2\} \subseteq B_1^{[1]}, $ and $ [x_2\p_2, \p_2]=-\p_2, $ one can prove it by induction.
	\end{enumerate}

\subsection{}\label{def of standard comp Borel} We introduce the following subspaces which will be useful.


$\bullet$ $\lieb_0=\lieb \ds W(n)_{\geq 1}$,  where $ \lieb $ consists of all upper triangular matrices of $  \frak{gl}(n)$.

$\bullet$ $\lieb_n=\mathfrak{t}_0 \ds C_n$,  where $C_n= \langle X^{\co {a}(i)}\p_i \mid \co {a}(i)=(a_1,  \cdots, a_{i-1}, 0, \cdots, 0), a_j\in I
\rangle, $ and $ \mathfrak{t}_0= \langle x_i\p _i \mid i=1, \cdots, n\rangle $ is the standard torus of $ W(n). $

$\bullet$ For arbitrary $ q=1,2,\cdots, n-1$,  let
$$  \lieb_q= \lieb_0(x_1,\cdots, x_{n-q}) \ds Q_q \ds  \lieb_q (x_{n-q+1}, \cdots, x_ n)= \mathfrak{t}_0 \ds C_q. $$
$$ Q_q=\langle u^{\co {a}(i)}w^{\co {b}(i)}\p _i \mid (\co {a}(i), \co {b}(i), i)\in \Gamma_1 \cup \Gamma_2\rangle,$$
\[ C_q=\langle u^{\co {a}(i)}w^{\co {b}(i)}\p _i \mid (\co {a}(i), \co {b}(i), i)\in \Lambda_1 \cup \Lambda_2\rangle, \]
where  $ (\co {a}(i), \co {b}(i), i):= (a_1, \cdots, a_{n-q}, b_1, \cdots, b_q,i)\in I ^{n-q}\times I ^q \times \{1,\cdots, n\}$,
\begin{align*}
\Gamma_1=&\{(\co {a}(i),\co {b}(i),i)  \mid \ 1\leq i\leq n-q, \abs{\co {b}(i)}>0, \text{either } \abs{\co {a}(i)}>1\text{ or }\ \abs{\co {a}(i)}=1=\sum_{k=1}^{i-1} a_k\},\\ 
\Gamma_2=&\{(\co {a}(i),\co {b}(i),i)  \mid n-q+1\leq i\leq n, \abs{\co {a}(i)}>0 \big\},\\
\Lambda_1=&\{(\co {a}(i),\co {b}(i),i)  \mid \ \ 1\leq i\leq n-q, \ \text{either}\ \abs{\co {a}(i)}>1\text{\ or\ }\abs{\co {a}(i)}=1=\sum_{k=1}^{i-1} a_k\},\\
\Lambda_2=&\{(\co {a}(i),\co {b}(i),i)  \mid n-q+1\leq i\leq n, \co {b}(i)=(b_1,\cdots, b_{i-q-1},0, \cdots)\ \text{once}\ \co {a}(i)=0 \big\}.
\end{align*}

\begin{remark} Note that $ \liet_0\subset \lieb_q,\ q=0,\cdots,n$. 
After proposition \ref{conj class of complete Borel}, we will see that all $ \lieb_r(r=0,\cdots,n) $ are torus graded.
\end{remark}

\begin{lemma}
	$  \lieb_0^{[1]}=\mathfrak{n}\ds W(n)_{\geq 1} $ is nilpotent, where $ \mathfrak{n} $ is  nilradical of $ \lieb. $ Moreover, $  \lieb_0 $ is a maximal completely solvable subalgebra.
\end{lemma}
\begin{proof}
	Since $ W(n)_{\geq 1} $ is a $ W(n)_0 $-module and $ [\lieb, \lieb]= \mathfrak{n}, $ $  \lieb_0^{[1]}=\mathfrak{n}\ds W(n)_{\geq 1} $ which is nilpotent. Moreover, \cite{S7} has proven that $  \lieb_0 $ is a Borel. Lemma holds.
\end{proof}

\begin{lemma}\label{bn}
	$  \lieb_n $ is a maximal completely solvable subalgebra.
\end{lemma}
\begin{proof}
	Thanks to the definition, $  \lieb_n $ contains a maximal torus $ \mathfrak{t}_0 $. 
	
	One can check that $  \lieb_n^{[1]}=C_n $ by definition. Moreover,
\begin{align*}
C_n ^{[1]} \subseteq&\langle X^{\co {a}(i)}\p_i\in C_n \mid  \co {a}(i)\neq (p-1, p-1, \cdots, p-1, 0) \rangle,\\ 
C_n ^{[p]} \subseteq&\langle X^{\co {a}(i)}\p_i\in C_n \mid  \co {a}(i)\neq (a_1, p-1, \cdots, p-1, 0), a_1\in I \rangle,
\end{align*}
\begin{align*}
C_n ^{[p^s]} \subseteq&\langle X^{\co {a}(i)}\p_i\in C_n \mid  \co {a}(i)\neq (a_1,\cdots,a_s, p-1, \cdots, p-1, 0), a_1, \cdots, a_s \in I \rangle
\end{align*}
	for arbitrary $ s. $
	
	Therefore, $ C_n ^{[p^{n-1}]}=0 $ and $ \lieb_n $ is completely solvable.
	
	\textbf{Maximality}: $ \forall D \not\in  \lieb_n, $ denote $ A_n $ be the algebra generated by $  \lieb_n $ and $ D. $ 

	Suppose $ D=\sum D_i, $ where $ D_i\in W(n)_i. $ 
	There must be $ D_k=\sum_{l=1}^n f_l\p _l\in A_n\xg \lieb_n $ by $ \liet_0 $-action. After proper $ \p _j $'s actions, we can assume $ D_k=X^{\co {a}(l)}\p_l\in A_n\backslash  \lieb_n, $
	where $ \co {a}(l) =(\co {a}(l)_1,\cdots,\co {a}(l)_n). $ Namely, $ \co {a}(l)_s\neq 0 $ for some $ s\geq l. $
	
	\begin{enumerate}
		\item  If $ \co {a}(l)_s\neq 0 $ for some $ s> l, $ we have $ x_s\p _l\in A_n, $ and hence the semisimple subalgebra $ \langle x_s\p _l, x_l\p _s, x_l\p _l-x_s\p _s  \rangle$ lies in $ A_n. $
		
		\item If $ \co {a}(l)_s = 0 $ whence $ s> l, $ and $ \co {a}(l)_l \geq 2, $ then $ x_l^2\p _l\in A_n. $ The semisimple subalgebra $ \langle \p _l, x_l\p _l, x_l^2\p _l  \rangle$ lies in $ A_n. $
		
		\item If $ \co {a}(l)_s = 0 $ whence $ s> l, $ then $ \co {a}(l)_l =1, $ and $ x_ix_l\p _l\in A_n\backslash  \lieb_n,\ i<l. $ Moreover,
		$ x_l\p _l=[\p _i, x_ix_l\p _l]\in A_n^{[1]},\ \p _l\in (A_n^{[1]})^{[t]}  $ for all $ t\in\fn. $
	\end{enumerate}
	
	Hence, $ A_n^{[1]} $ is not nilpotent. Namely, $  \lieb_n $ is maximal.
\end{proof}

\begin{lemma}
	$  \lieb_q $ is a maximal completely solvable subalgebra for $ q=1,\cdots,n-1 $.
\end{lemma}
\begin{proof}
	Thanks to the definition, $  \lieb_q $ contains a maximal torus $ \mathfrak{t}_0 $ and $  \lieb_q^{[1]}=C_q. $
	
	Recall that $ \lieb_q=\liet_0\ds C_{q1}\ds C_{q2}=\lieb_0(x_1,\cdots,x_{n-q})\ds\liet_0\ds \lieb_q(x_{n-q+1},\cdots,x_n), $ where $ C_{qi} $ corresponding to $ \Lambda_i $ as section 3.2 ($ i=1,2 $). Then the followings holds.	
\begin{align*}
C_q^{[1]}  \subseteq&C_{q2}\ds \langle u^{\co {a}(i)}w^{\co {b}(i)}\p _i\in C_{q1}  \mid  u^{\co {a}(i)}w^{\co {b}(i)}\p _i\neq x_{n-q-1} x_n^{p-1} \p _{n-q} \rangle;\\ 
C_q^{[p]}  \subseteq& C_{q2}\ds \langle u^{\co {a}(i)}w^{\co {b}(i)}\p _i\in C_{q1}  \mid  u^{\co {a}(i)}w^{\co {b}(i)}\p _i\neq x_{n-q-1} x_n^l \p _{n-q},l=0,\cdots,p-1 \rangle;\\
C_q^{[p^q]}  \subseteq& C_{q2}\ds \langle u^{\co {a}(i)}w^{\co {b}(i)}\p _i\in C_{q1}  \mid  \co {a}(i)\neq \epsilon_{n-q-1} \rangle.
\end{align*}
	Suppose $ ( \lieb_0^{[1]}(x_1,\cdots,x_{n-q}))^{[s]}=0$ and $ ( \lieb_q^{[1]}(x_{n-q+1},\cdots,x_n))^{[t]}=0. $ Then 
\begin{align*}
C_q^{[sp^q]}  \subseteq& C_{q2}, \\ 
C_q^{[t+sp^q]}\subseteq& \langle u^{\co {c}(j)}w^{\co {d}(j)}\p _j\in C_{q2}  \mid  \co {c}(j)\neq 0 \rangle;\\ 
C_q^{[t+sp^q+(n-q)p^q]}\subseteq &\langle u^{\co {c}(j)}w^{\co {d}(j)}\p _j\in C_{q2} \mid  \co {c}(j)\neq 0,\epsilon_1,\cdots, \epsilon_{n-q} \rangle;\\
C_q^{[t+sp^q+p^{n-q}p^q]}=  \subseteq& C_q^{[t+sp^q+p^n]}=0
\end{align*}
    Hence, $ C_q $ is nilpotent and $ \lieb_q $ is completely solvable.
	
	\textbf{Maximality}: Similar to the maximality of  $  \lieb_n, $ suppose $ A_q $ contains $  \lieb_q $ as a proper subset, then either there is an element $ \p _l $ lies in $ (A_q^{[1]})^{[m]} $ for all $ m\in\fn $ which violates the nilpotence of $ A_q^{[1]} $ or there exists a semisimple subalgebra in $ A_q, $ say either $ \langle x_s\p _l, x_l\p _s, x_l\p _l-x_s\p _s  \rangle (l\neq s)$  or $ \langle \p _l, x_l\p _l, x_l^2\p _l  \rangle.$
\end{proof}

\subsection{Conjugation classes of maximal conpletely solvable subalgebras}
By using similar idea and methods of \cite{S7}, we will classify conjugation classes of all maximal complete solvable subalgebras which are torus graded. 

Let $ \lieh $ be any subalgebra of $ W(n) $. Similar to \cite{S7}, define
$$\on{r}( \lieh) := \text{max}\{r \mid  \text{there\ exists}\ \sigma\in \text{Aut}(W(n))\ \text{such\ that }\ \sigma(t_r) \subset  \lieh\},$$
and $ \on{r}(\lieh)=-1 $ if $ \lieh $ does not contain any maximal torus.

The following lemma is easy to verify.
\begin{lemma} Keep notations as above. For every $ r=0,\cdots,n, $ 
	\begin{enumerate}
		\item  $ \liet_i\subseteq\lieb_r $ if and only if $ 0\leq i\leq r. $
		\item 		$ \on{r}( \lieb_r) = r. $
	\end{enumerate}
\end{lemma}

The arguments of \cite[lemma 4.2-4.4]{S7} still work in our case. However, one need to check that all of the algebras involved are completely solvable not only solvable. Thus we have

\begin{proposition}\label{conj class of complete Borel}
	Assume $ p>3. $ Let $ \lieb $ be a torus graded and maximal completely solvable subalgebra of $ W(n) $ with $  \on{r}( \lieb)=r $. Then $ \lieb $ conjugates to $  \lieb_r $ with respect to $ G. $ 
	
	In particular, each $ \lieb_r( r=0,\cdots,n) $ is torus graded and maximal complete solvable.
\end{proposition}

\begin{corollary}
	Let $ \lieb $ be a torus graded and maximal completely solvable subalgebra of $ W(n), $ then $$ \on{r}( \lieb)=\mathsf{dim}(\mathsf{pr}_0( \lieb)), $$ where $ \mathsf{pr}_0: W(n)\sur W(n)/W(n)_{\geq 0} $ is a linear map.
\end{corollary}
\begin{proof}
	$ \mathsf{pr}_0( \lieb_r)=<\p _{i} \mid  i=n-r+1,\cdots,n>, $ hence $ \on{r}( \lieb_r)=r=\mathsf{dim}(\mathsf{pr}_0( \lieb_r)). $
	
	Note that $ \mathsf{pr}_0(\sigma( \lieb_r))=\mathsf{pr}_0( \lieb_r) $ and $ \mathsf{dim}(\mathsf{pr}_0(g( \lieb)))=\mathsf{dim}(\mathsf{pr}_0( \lieb)) $ for all $ \sigma\in U, $ $ g\in G_0. $ Corollary holds.
\end{proof}

\begin{remark}
	Thanks to proposition \ref{conj class of complete Borel}, the conjugation class of $ \lieb $ could be determined by $ \mathsf{dim}(\mathsf{pr}_0( \lieb)). $	
\end{remark}

Since $ \mathsf{dim}(\mathsf{pr}_0( \lieb_r))=r,$
the following holds.
\begin{corollary}
	$  \lieb_0,\cdots, \lieb_n  $ are not conjugate with each other.
\end{corollary}

Moreover, we have the following classification theorem.

\begin{theorem}\label{main thm of conjugation class}
	Assume $ p>3. $ There are $ (n+1) $ conjugacy classes of torus graded and maximal completely solvable subalgebras of $ W(n)$ with representatives $\{ \lieb_i\mid i = 0, \cdots , n, \}. $
\end{theorem}

\subsection{Completely solvable subalgebras with standard grading}
Recall that a subalgebra $ \lieh $ is of standard grading if $ \lieh=\sum_i \lieh\cap W(n)_i. $

By definition, all $ \lieb_r $'s are torus graded and maximal completely solvable subalgebras of $ W(n) $ with standard grading. 

\begin{theorem}\label{homo}
	Assume $ p>3. $ Let $ \lieb $ be a torus graded and maximal completely solvable subalgebra of $ W(n) $ with standard grading and $  \on{r}( \lieb)=r $. Then $ \lieb $ conjugates to $  \lieb_r $ with respect to $ G. $ 
	
	Namely, there are $ (n+1) $ conjugacy classes of torus graded and maximal completely solvable subalgebras of $ W(n) $ with standard grading. The representatives are $ \{\lieb_0,\cdots,\lieb_n\}. $
\end{theorem}
\begin{proof}
Thanks to proposition \ref{conj class of complete Borel}, there is $ \sigma\in U, g\in G_0 $ such that $ \lieb=\sigma\cdot g\cdot \lieb_r. $ Now, apply the gradation functor on both side, $ \lieb=\mathsf{gr}(\lieb)= \mathsf{gr}(\sigma\cdot g\cdot \lieb_r)=g\cdot \lieb_r. $
\end{proof}

\section{(Graded) dimensions}
Recall that $ \mathsf{gdim}(L) $ is the graded dimension for a filtered algebra $ L $ (definition \ref{gdim}). We have proved that $ \mathsf{gdim}(L) $ is $ G $-invariant if $ L $ and $ G $ satisfies the graded assumption (definition \ref{grading assumption}).

\begin{lemma}\label{com Borels are filtered}
	All torus graded and maximal completely solvable subalgebras of $ W(n) $ are filtered subalgebras.
\end{lemma}
\begin{proof}
	Note that $( W(n), (W(n)_i), G,G_0,U) $ satisfies the graded assumptions in Lemma \ref{filtration}, and all $ \lieb_r,\ r=0,\cdots,n, $ are graded. Claim holds.
\end{proof}

\begin{proposition}
	Suppose $ \lieb $ is a torus graded and maximal completely solvable subalgebra of $ W(n) $ with $ \on{r}( \lieb)=r, $ then
	$$ \mathsf{gdim}( \lieb)= n(Q^{n-r}-1)Q^rt^{-1} +\frac{1-Q^r}{1-Q}t^{-1} - \frac{(n-r)(n+r+1)}{2}(Q^r-1) $$
	$$+r-\frac{(n-r)(n-r-1)}{2}, $$
	where $ Q=\sum_{i=0}^{p-1} t^i. $
\end{proposition}
\begin{proof}
	
	One can check that $
	\mathsf{gdim}(A(n))=Q^n\ \text{and }
	\mathsf{gdim}(A(n)\p_i)= Q^nt^{-1},	$ where
	$
	Q:=\mathsf{gdim}(A(1))=\sum_{i=0}^{p-1} t^i.
	$
Then $ \mathsf{gdim}(\lieb_0)= nQ^nt^{-1} -nt^{-1}-\frac{n(n-1)}{2} . $

For $ \lieb_n $ case,	
	recall that $ \lieb_n=\liet_0+ \langle X^{\un{a}}\p_i \mid \un{a}=(a_1,\cdots,a_{i-1},0,\cdots, 0) \rangle.$ 
	Therefore,
		$$ \mathsf{gdim}(\lieb_n)   = \sum_{i=1}^{n}\mathsf{gdim}(\lieb_n\cap A(n)\p_i) 
		= \frac{1-Q^n}{1-Q}t^{-1}+n. $$
	
	
	Now, for arbitrary $ q\neq 0, n, $
	$ \lieb_q= \lieb_0(x_1,\cdots, x_{n-q}) \ds Q_q \ds  \lieb_q (x_{n-q+1}, \cdots, x_ n), $
	where $ Q_q=\langle u^{\co {a}(i)}w^{\co {b}(i)}\p _i \mid (\co {a}(i), \co {b}(i), i)\in \Gamma_1 \cup \Gamma_2\rangle $ as \ref{def of standard comp Borel}.

	Set
	$ R_i:=Q_q\cap A(n)\p_i,\ i=1,\cdots,n. $ We have the followings.
	
	If $ 1\leq i \leq n-q, $
		$ \mathsf{gdim}(R_i) = (Q^{n-q}-1)(Q^q-1)t^{-1}-(n-i+1)(Q^q-1); $
	
	If $ n-q+1\leq i \leq n, $
		$ \mathsf{gdim}(R_i) = (Q^{n-q}-1)Q^qt^{-1}. $
	
	Moreover,	
	
	\noindent\begin{tabular}{rcl}
		$ \mathsf{gdim}(\lieb_q) $ 
		& = & $ n(Q^{n-q}-1)Q^qt^{-1} +\frac{1-Q^q}{1-Q}t^{-1} - \frac{(n-q)(n+q+1)}{2}(Q^q-1) $ \\
		& & $+q-(n-q)(n-q-1)/2. $
	\end{tabular}
	
	
	Now, suppose $ \lieb $ is an arbitrary torus graded and maximal completely solvable subalgebra of $ W(n) $ with $ \on{r}(\lieb)=r. $ Then there exits $ \Phi\in G $ such that $ \Phi\cdot \lieb= \lieb_r. $ Moreover, $ \mathsf{gdim}(\lieb)=\mathsf{gdim}(\Phi\cdot\lieb)=\mathsf{gdim}( \lieb_r). $
\end{proof}

\begin{remark}
	By using the formula $ Q(0)=1 $ and $ \frac{dQ}{dt}|_{t=0}=1, $ one can get
	$$
	\mathsf{dim}(\lieb_{r,-1})=r,
	$$
	$$
	\mathsf{dim}(\lieb_{r,0})=r(r+1)/2,
	$$
	which match the fact that
	$$ \lieb_r=\sum_{i=n-r+1}^{n}\fk \p_i+ \sum_{i\leq j}\fk x_i\p_j +\text{ higher degree terms}. $$
\end{remark}

Note that $ Q(1)=p. $ We have the following corollary by taking $ t=1 $.
\begin{corollary}\label{dim of b_q}
	Suppose $ \lieb $ is a torus graded and maximal completely solvable subalgebra with $ \on{r}( \lieb)=r, $ then
	$$  \mathsf{dim}(\lieb)=n(p^{n-r}-1)p^r +\frac{1-p^r}{1-p}- \frac{(n-r)(n+r+1)}{2}(p^r-1) +r-\frac{(n-r)(n-r-1)}{2}.  $$
\end{corollary}


It is easy to verify that all homogeneous Borel subalgebras defined  of $ W(n) $ in \cite{S7} are filtered and hence one can compute their (graded) dimensions as well.

\begin{proposition}
	Suppose $ B $ is a homogeneous Borel subalgebra with $ \on{r}( B)=q, $ then
	$$ \mathsf{gdim}( B)= nQ^nt^{-1}-nQ^qt^{-1}-\frac{(n-q)(n-q-1)}{2} Q^q +(1+t^{-1}) \frac{1-Q^q}{1-Q}, $$
	where $ Q=\sum_{i=0}^{p-1} t^i. $ And
	\[ \mathsf{dim}( B)= np^nt^{-1}-np^q-\frac{(n-q)(n-q-1)}{2} p^q + \frac{2(1-p^q)}{1-p}. \]
\end{proposition}




\end{document}